\documentclass[12pt]{amsart}

\usepackage{amssymb}
\usepackage{latexsym, amscd, verbatim, graphics, ulem, color}

\newtheorem{theorem}{Theorem}[section]
\newtheorem{proposition}[theorem]{Proposition}
\newtheorem{lemma}[theorem]{Lemma}
\newtheorem{corollary}[theorem]{Corollary}

\theoremstyle{definition}

\newtheorem{remark}[theorem]{Remark}

\newtheorem{conjecture}[theorem]{Conjecture}

\newcommand{\ZZ}{ \ensuremath{\mathbb{Z}}}
\newcommand{\Gin}{\ensuremath{\mathrm{Gin}}}

\newcommand{\init}{\ensuremath{\mathrm{in}}}

\newcommand{\GL}{{GL}}
\newcommand{\SL}{{SL}}

\newcommand{\Tor}{\ensuremath{\mathrm{Tor}}}

\newcommand{\lex}{{\mathrm{lex}}}
\newcommand{\rlex}{{\mathrm{{oplex}}}}
\newcommand{\rev}{{\mathrm{{rev}}}}

\def\cocoa{{\hbox{\rm C\kern-.13em o\kern-.07em C\kern-.13em o\kern-.15em A}}}

\newcommand{\ww}{\mathbf w}
\newcommand{\oplex}{{\mathrm{{oplex}}}}

\begin{document}
\title[Hilbert functions of general intersections]
{On Hilbert functions of\\ general intersections of ideals}

\author{Giulio Caviglia}
\address{
Department of Mathematics,
Purdue University,
West Lafayette,
IN 47901, USA.
}
\email{gcavigli@math.purdue.edu}

\author{Satoshi Murai}
\address{
Satoshi Murai,
Department of Mathematical Science,
Faculty of Science,
Yamaguchi University,
1677-1 Yoshida, Yamaguchi 753-8512, Japan.
}
\email{murai@yamaguchi-u.ac.jp}

\thanks{The work of the first author was supported by a grant from the
Simons Foundation (209661 to G. C.).
The work of the second author was supported by KAKENHI 22740018.
}

\maketitle

\begin{abstract}
Let $I$ and $J$ be homogeneous ideals in a standard graded polynomial ring.
We study upper bounds of the Hilbert function of the intersection of $I$ and $g(J)$,
where $g$ is a general change of coordinates.
Our main result gives a generalization of Green's hyperplane section theorem.
\end{abstract}

\section{Introduction}
Hilbert functions of graded $K$-algebras
are important invariants studied in several areas of mathematics.
In the theory of Hilbert functions,
one of the most useful tools is Green's hyperplane section theorem,
which gives a sharp upper bound for the Hilbert function of $R/hR$, where $R$ is a standard graded $K$-algebra  and $h$ is a general linear form, in terms of the Hilbert function of $R$.
This result of Green has been extended to the case of general homogeneous polynomials
by Herzog and Popescu \cite{HP} and Gasharov \cite{Ga}.
In this paper, we study a further generalization of these theorems.

Let $K$ be an infinite field
and $S=K[x_1,\dots,x_n]$ a standard graded polynomial ring.
Recall that the \textit{Hilbert function} $H(M,-) : \mathbb{Z} \to \mathbb{Z}$ of a finitely generated graded $S$-module $M$
is the numerical function defined by
$$H(M,d)=\dim_K M_d,$$
where $M_d$ is the graded component of $M$ of degree $d$.
A set $W$ of monomials of $S$ is said to be \textit{lex}
if, for all monomials $u,v \in S$ of the same degree,
$u \in W$ and $v>_\lex u$ imply $v \in W$,
where $>_\lex$ is the lexicographic order induced by the ordering $x_1> \cdots > x_n$.
We say that a monomial ideal $I \subset S$ is a \textit{lex ideal}
if the set of monomials in $I$ is lex.
The classical Macaulay's theorem \cite{Ma} guarantees that,
for any homogeneous ideal $I \subset S$,
there exists a unique lex ideal, denoted by $I^{\lex}$, with the same Hilbert function as $I$.
Green's hyperplane section theorem \cite{Gr} states

\begin{theorem}[Green's hyperplane section theorem]
\label{green}
Let $I \subset S$ be a homogeneous ideal.
For a general linear form $h \in S_1$,
$$H(I \cap (h),d) \leq  H(I^\lex \cap (x_n),d) \ \ \mbox{for all } d \geq 0.$$
\end{theorem}


Green's hyperplane section theorem is known to be useful to prove several important results on Hilbert functions such as Macaulay's theorem \cite{Ma} and Gotzmann's persistence theorem \cite{Go}, see \cite{Gr}.
Herzog and Popescu \cite{HP} (in characteristic $0$) and Gasharov \cite{Ga} (in positive characteristic) generalized Green's hyperplane section theorem in the following form.

\begin{theorem}[Herzog--Popescu, Gasharov]
\label{hpg}
Let $I \subset S$ be a homogeneous ideal.
For a general homogeneous polynomial $h \in S$ of degree $a$,
$$H(I \cap (h),d) \leq  H(I^\lex \cap(x_n^a),d) \ \ \mbox{for all } d \geq 0.$$
\end{theorem}

We study a generalization of Theorems \ref{green} and \ref{hpg}.
Let $>_\rlex$ be the lexicographic order on $S$ induced by the ordering $x_n> \cdots > x_1$.
A set $W$ of monomials of $S$ is said to be \textit{opposite lex}
if, for all monomials $u,v \in S$ of the same degree,
$u \in W$ and $v>_\rlex u$ imply $v \in W$.
Also, we say that a monomial ideal $I \subset S$ is an \textit{opposite lex ideal}
if the set of monomials in $I$ is opposite lex.
For a homogeneous ideal $I \subset S$,
let $I^\oplex$ be the opposite lex ideal with the same Hilbert function as $I$
and let $\Gin_\sigma(I)$ be the generic initial ideal (\cite[\S 15.9]{Ei}) of $I$
with respect to a term order $>_\sigma$.

In Section 3 we will prove the following

\begin{theorem}
\label{intersection} Suppose  $\mathrm{char}(K)=0$.
Let $I\subset S$ and $J \subset S$ be homogeneous ideals such that $\Gin_\lex(J)$ is lex.
For  a general change of coordinates $g$ of $S$,
$$H(I \cap g(J),d) \leq H(I^\lex \cap J^\oplex ,d)
\ \ \mbox{for all } d\geq 0.$$
\end{theorem}

Theorems \ref{green} and \ref{hpg}, assuming that the characteristic is zero, are special cases of the above theorem  when $J$ is principal.
Note that Theorem \ref{intersection} is sharp since the equality holds if $I$ is lex and $J$ is oplex (Remark \ref{rem1}). Note also that if $\Gin_\sigma(I)$ is lex for some term order $>_\sigma$ then $\Gin_\lex(J)$ must be lex as well (\cite[Corollary 1.6]{Co1}). 

Unfortunately, the assumption on $J$, as well as the assumption on the characteristic of $K$,  in Theorem \ref{intersection}
are essential (see Remark \ref{example}).
However, we prove the following result for the product of ideals.

\begin{theorem}
\label{product}
Suppose $\mathrm{char}(K)=0$.
Let $I\subset S$ and $J \subset S$ be homogeneous ideals.
For  a general change of coordinates $g$ of $S$,
$$H(I  g(J),d) \geq H(I^\lex J^\oplex ,d)
\ \ \mbox{for all } d\geq 0.$$
\end{theorem}

Inspired by Theorems \ref{intersection} and \ref{product},
we suggest the following conjecture.

\begin{conjecture}
\label{conj} Suppose $\mathrm{char}(K)=0.$
Let $I\subset S$ and $J \subset S$ be homogeneous ideals such that $\Gin_\lex(J)$ is lex.
For  a general change of coordinates $g$ of $S$,
\[
\dim_K \Tor_i(S/I,S/g(J))_d \leq  \dim_K \Tor_i(S/I^\lex,S/J^\oplex)_d
\ \ \mbox{for all } d\geq 0.
\]
\end{conjecture}

Theorems \ref{intersection} and \ref{product}
show that the conjecture is true if $i=0$ or $i=1.$ 
The conjecture is also known to be true when $J$ is generated by linear forms  by the result of Conca \cite[Theorem 4.2]{Co}. Theorem \ref{2.5}, which we prove later, also provides some evidence supporting the above inequality.


\section{Dimension of $\Tor$ and general change of coordinates}

Let $\GL_n(K)$ be the general linear group of invertible $n \times n$ matrices over $K$.
Throughout the paper, we identify each element $h=(a_{ij}) \in \GL_n(K)$
with the change of coordinates defined by $h(x_i)=\sum_{j=1}^n a_{ji}x_j$
for all $i$.

We say that a property (P) holds for a general $g \in \GL_n(K)$
if there is a non-empty Zariski open subset $U \subset \GL_n(K)$
such that (P) holds for all $g \in U$.

We first prove that, for two homogeneous ideals $I \subset S$ and $J \subset S$,
the Hilbert function of $I \cap g(J)$ and that of $I g(J)$ are  well defined for a general $g \in \GL_n (K)$, i.e.\ there exists a non-empty Zariski open subset of $\GL_n(K)$ on which the Hilbert function of $I \cap g(J)$ and that of $I g(J)$ are constant.
\begin{lemma}
\label{2-0}
Let $I \subset S$ and $J \subset S$ be homogeneous ideals.
For a general change of coordinates $g \in \GL_n(K)$,
the function $H(I \cap g(J),-)$
and $H(I g(J),-)$ are well defined.
\end{lemma}
\begin{proof}
We prove the statement for $I \cap g(J)$
(the proof for $Ig(J)$ is similar).
It is enough to prove the same statement for $I+g(J)$.
We prove that $\init_\lex(I+g(J))$ is constant for a general $g \in \GL_n(K)$.

Let $t_{kl}$, where $1 \leq k,l \leq n$, be indeterminates,
$\tilde K=K(t_{kl}: 1 \leq k,l \leq n)$ the field of fractions of $K[t_{kl}: 1 \leq k,l \leq n]$
and $A=\tilde K [x_1,\dots,x_n]$. 
Let $\rho: S \to A$ be the ring map induced by $\rho(x_k)= \sum_{l=1}^n t_{lk} x_l$ for $k=1,2,\dots,n$,
and  $\tilde L= I A + \rho(J)A \subset A$.
Let $L \subset S$ be the monomial ideal with the same monomial generators
as $\init_\lex(\tilde L)$.
We prove $\init_\lex(I+g(J))=L$ for a general $g \in \GL_n(K)$.

Let $f_1,\dots,f_s$ be generators of $I$
and $g_1,\dots,g_t$ those of $J$.
Then the polynomials $f_1,\dots,f_s,\rho(g_1),\dots,\rho(g_t)$ are generators of $\tilde L$.
By the Buchberger algorithm,
one can compute a Gr\"obner basis of $\tilde L$ from
$f_1,\dots,f_s,\rho(g_1),\dots,\rho(g_t)$ by finite steps.
Consider all elements $h_1,\dots,h_m \in K(t_{kl}:1 \leq k,l \leq n)$
which are the coefficient of polynomials (including numerators and denominators of rational functions)
that appear in the process of computing a Gr\"obner basis of $\tilde L$ by the Buchberger algorithm.
Consider a non-empty Zariski open subset $U \subset \GL_n(K)$
such that $h_i(g) \in K \setminus \{0\}$ for any $g \in U$,
where $h_i(g)$ is an element obtained from $h_i$ by substituting $t_{kl}$ with entries of $g$.
By construction
$\init_\lex(I+g(J))=L$
for every $g \in U$.
\end{proof}

\begin{remark}\label{ConstantHF} The method used to prove the above lemma can be easily generalized to a number of situations.
For instance for a general $g \in \GL_n(K)$ and a finitely generated graded $S$-module $M,$ 
the Hilbert function of $\Tor_i(M,S/g(J))$ is  well defined for every $i$. Let  
$\mathbb F: 0 \stackrel{\varphi_{p+1}}{\longrightarrow} 
\mathbb F_p \stackrel{\varphi_p}{\longrightarrow} 
\cdots 
\longrightarrow
\mathbb F_1 \stackrel{\varphi_1}{\longrightarrow} 
\mathbb F_0
\stackrel{\varphi_0}{\longrightarrow}0$ 
be a graded free resolution of $M.$ Given a change of coordinates $g$, one first notes that for every $i=0,\dots,p$, the Hilbert function $H(\Tor_i(M,S/g(J)),-)$ is equal to the difference between the Hilbert function of $\rm{Ker}(\pi_{i-1} \circ \varphi_i)$ and the one of $\varphi_{i+1}(F_{i+1}) + F_i \otimes_S g(J)$ where $\pi_{i-1}: F_{i-1} \rightarrow F_{i-1} \otimes_S S/g(J)$ is the canonical projection. 
Hence we have
\begin{align}\label{H-TOR}
\nonumber H(\Tor_i & (M,S/g(J)),-)= \\
 &H(F_i, -) -H(\varphi_i(F_i)+    g(J) F_{i-1},-)
 + H(g(J) F_{i-1},-)\\
\nonumber &- H(\varphi_{i+1}(F_{i+1}) + g(J) F_i,-).
\end{align}
Clearly  $H(F_i,-)$ and $H(g(J) F_{i-1},-)$ do not depend on $g.$
Thus it is enough to show that, for a general $g$, the Hilbert functions of $\varphi_i(F_i)+g(J) F_{i-1}$ are well defined for all $i=0,\dots,p+1.$ This can be seen as in Lemma \ref{2-0}.
\end{remark}

Next, we present two lemmas which will allow us to reduce the proofs of the theorems in the third section to  combinatorial considerations regarding Borel-fixed ideals.

The first Lemma is probably clearly true to some experts,
but we include its proof for the sake of the exposition.
The ideas used in Lemma \ref{lemma2} are similar to that of \cite[Lemma 2.1]{Ca1} and they rely on the construction of a flat family and  on the use of the structure theorem for finitely generated modules over principal ideal domains. 

\begin{lemma}
\label{lemma1}
Let $M$ be a finitely generated graded $S$-module
and $J \subset S$ a homogeneous ideal.
For a general change of coordinates $g \in \GL_n(K)$ we have that
$\dim_K \Tor_i(M,S/g(J))_j \leq \dim_K \Tor_i(M,S/J)_j$ for all $i$ and for all $j.$
\end{lemma}

\begin{proof} 
Let  $\mathbb F$ be a resolution of $M,$ as in Remark \ref{ConstantHF}. Let $i$, $0\leq i \leq p+1$ and notice that, by equation \eqref{H-TOR}, it is sufficient to show: $H(\varphi_i(F_i)+g(J) F_{i-1},-)\geq 
H(\varphi_i(F_i)+JF_{i-1},-).$ We fix a degree $d$ and consider the monomial basis of $ (F_{i-1})_d.$ 
Given a change of coordinates $h=(a_{kl}) \in \GL_n(K)$ we present the vector space $V_d=(\varphi_i(F_i)+h(J)F_{i-1})_d$ with respect to this basis. The dimension of $V_d$ equals the rank of a matrix whose entries are polynomials in the  $a_{kl}$'s with coefficients in $K.$ Such a rank is maximal when the change of coordinates $h$ is general.
\end{proof}


For a vector $\ww=(w_1,\ldots,w_n) \in \ZZ_{\geq 0}^n$,
let $\init_\ww (I)$ be the initial ideal of a homogeneous ideal $I$ with respect to the
weight order $>_\ww$
(see \cite[p.\ 345]{Ei}).
Let $T$ be a new indeterminate and 
$R=S[T]$.
For $\mathbf a=(a_1,\dots,a_n) \in \ZZ_{\geq 0}^n$,
let $x^{\mathbf a}=x_1^{a_1} x_2^{a_2} \cdots x_n^{a_n}$
and $(\mathbf a, \ww)= a_1w_1 + \cdots + a_n w_n$.
For a polynomial $f= \sum_{\mathbf a \in \ZZ_{\geq 0}^n} c_{\mathbf a} x^{\mathbf a}$,
where $c_{\mathbf a} \in K$,
let $b= \max \{ (\mathbf a,\ww) : c_{\mathbf a} \ne 0\}$
and
$$\tilde f = T^b \left(\sum_{\mathbf a \in \ZZ_{\geq 0}^n} T^{-(\mathbf a,\ww)}c_{\mathbf a} x^{\mathbf a}\right) \in R.$$
Note that $\tilde f$ can be written as $\tilde f=\init_\ww(f) + T g$ where $g \in R$.
For an ideal $I \subset S$,
let $\tilde I =(\tilde f :f \in I) \subset R$.
For $\lambda \in K \setminus\{0\}$,
let $D_{\lambda,\ww}$ be the diagonal change of coordinates defined by $D_{\lambda,\ww}(x_i)=\lambda^{-w_i} x_i$.
From the definition, we have
$$R/\big(\tilde I +(T)\big) \cong S/ \init_\ww(I)$$
and
$$R/\big(\tilde I +(T-\lambda)\big) \cong S/D_{\lambda,\ww}(I)$$
where $\lambda \in K \setminus \{0\}$.
Moreover $(T-\lambda)$ is a non-zero divisor of $R/\tilde I$ for any $\lambda \in K$.
See \cite[\S 15.8]{Ei}.

\begin{lemma}
\label{lemma2}
Fix an integer $j$.
Let $\ww \in \ZZ_{\geq 0}^n$,
$M$ a finitely generated graded $S$-module and $J \subset S$ a homogeneous ideal.
For a general $\lambda \in K$, one has
\[
\dim_K \Tor_i \big(M,S/\init_\ww(J)\big)_j\geq \dim_K \Tor
_i\big(M, S/D_{\lambda,\ww}(J)\big)_j
\ \mbox{ for all $i$.}
\]
\end{lemma}

\begin{proof} 
Consider the ideal $\tilde {J} \subset R$ defined as above.
Let $\tilde M = M \otimes_S R$ and
$T_i=\Tor_i^{R}(\tilde M,R/\tilde{J})$.
By the structure theorem for
modules over a PID (see \cite[p.\ 149]{La}),
we have
$$(T_i)_j\cong K[T]^{a_{ij}}
\bigoplus A_{ij}$$
as a finitely generated $K[T]$-module,
where $a_{ij} \in \ZZ_{\geq 0}$ and where $A_{ij}$ is the torsion submodule.
Moreover $A_{ij}$ is a module of the form
$$A_{ij}\cong \bigoplus_{h=1}^{b_{ij}} K [T]/(P^{i,j}_{h}),$$
where $P^{i,j}_h$ is a non-zero polynomial in $K[T]$.
Set $l_{\lambda}=T-\lambda$.
Consider the exact sequence
\begin {eqnarray}
\label{aa}
\begin{CD} 0 @>>> R/\tilde{J} @>\cdot l_{\lambda}>>
R/\tilde{J} @>>> R/\big((l_{\lambda})+\tilde{J} \big) @>>> 0.
\end{CD}
\end {eqnarray}
By considering the long exact sequence induced by $\Tor^R_i(\tilde M,-),$ we have the following exact sequence
\begin{equation}\label{bo} 0\longrightarrow T_i/l_{\lambda} T_i \longrightarrow
\Tor_i^{R}\big(\tilde M,R/\big((l_{\lambda})+\tilde{J}\big)\big) \longrightarrow
K_{i-1} \longrightarrow 0,
\end{equation}
where $K_{i-1}$ is the kernel of the map $T_{i-1} \xrightarrow{\cdot l_{\lambda}} T_{i-1}$.
Since $l_{\lambda}$ is a regular element for $R$ and $\tilde M$,
the middle term in (\ref{bo}) is isomorphic to
\begin{eqnarray*}
\Tor_i^{R/(l_\lambda)} \big(\tilde M /l_\lambda \tilde M, R/\big((l_{\lambda})+\tilde J \big)\big)
=\left\{
\begin{array}{lll}
\Tor_i^S \big(M,S/\init_\ww(J)\big), & \mbox{ if } \lambda=0,\\
\Tor_i^S \big(M,S/D_{\lambda,\ww}(J)\big), & \mbox{ if } \lambda\ne0
\end{array}
\right.
\end{eqnarray*}
(see \cite[p.\ 140]{Mat}).
By taking the graded component of degree $j$ in (\ref{bo}),
we obtain
\begin{eqnarray}
\label{banngou} 
\begin{array}{lll}
\dim_K
\Tor_i^{S}\big(M,S/\init_\ww (J) \big)_j &=& a_{ij} + 
\# \{P^{ij}_h : P^{i,j}_h(0)=0\}\\
&& +  \# \{P^{i-1,j}_h : P^{i-1,j}_h(0)=0\},
\end{array}
\end{eqnarray}
where $\# X$ denotes the cardinality of a finite set $X$,
and
\begin{eqnarray}
\label{yon}
\dim_K
\Tor_i^{S}\big(M,S/D_{\lambda,\ww}(J) \big)_j &=& a_{ij}
\end{eqnarray}
for a general $\lambda \in K$.
This proves the desired inequality.
\end{proof}

\begin{corollary}
\label{add}
With the same notation as in Lemma \ref{lemma2},
for a general $\lambda \in K$,
\[
\dim_K \Tor_i \big(M,\init_\ww(J)\big)_j \geq \dim_K \Tor_i \big(M, D_{\lambda,\ww}(J) \big)_j
\mbox{ for all }i.
\]
\end{corollary}

\begin{proof}
For any homogeneous ideal $I \subset S$,
by considering the long exact sequence induced by $\Tor_i(M,-)$  from the short exact sequence
$0 \longrightarrow I \longrightarrow S \longrightarrow S/I \longrightarrow 0$ we have
$$\Tor_i(M,I) \cong \Tor_{i+1}(M,S/I)
\mbox{ for }i \geq 1$$
and
$$\dim_K \Tor_0(M,I)_j = \dim_K \Tor_1(M,S/I)_j + \dim_K M_j - \dim_K \Tor_0(M,S/I)_j.$$
Thus by Lemma \ref{lemma2} it is enough to prove that
\begin{eqnarray*}
&&\dim_K \Tor_1\big(M,S/\init_\ww(J)\big)_j -\dim_K \Tor_1\big(M,S/D_{\lambda,\ww}(J)\big)_j\\
&&\geq \dim_K \Tor_0\big(M,S/\init_\ww(J)\big)_j -\dim_K \Tor_0\big(M,S/D_{\lambda,\ww}(J)\big)_j.
\end{eqnarray*}
This inequality follows from (\ref{banngou}) and (\ref{yon}).
\end{proof}

\begin{proposition}
\label{2.3}
Fix an integer $j$.
Let $I \subset S$ and $J \subset S$ be homogeneous ideals.
Let $\ww,\ww' \in \ZZ_{\geq 0}^n$.
For a general change of coordinates $g \in \GL_n(K)$,
\begin{itemize}
\item[(i)]
$\dim_K \Tor_i(S/I,S/g(J))_j
\leq \dim_K \Tor_i (S/\init_{\ww}(I), S/{\init_{\ww'}}(J))_j
\ \mbox{ for all }i.$
\item[(ii)]
$\dim_K \Tor_i(I,S/g(J))_j
\leq \dim_K \Tor_i (\init_{\ww}(I), S/{\init_{\ww'}}(J))_j
\ \mbox{ for all }i.$
\end{itemize}
\end{proposition}

\begin{proof}
We prove (ii) (the proof for (i) is similar).
By Lemmas \ref{lemma1} and \ref{lemma2} and Corollary \ref{add}, we have
\begin{eqnarray*}
\dim_K \Tor_i \big(\init_{\ww}(I), S/\init_{\ww'}(J)\big)_j 
&\geq& \dim_K \Tor_i \big(D_{\lambda_1,\ww}(I), S/D_{\lambda_2,\ww'}(J)\big)_j \\
&=& \dim_K \Tor_i \big(I, S/D^{-1}_{\lambda_1,\ww} \big(D_{\lambda_2,\ww'}(J)\big)\big)_j\\
&\geq& \dim_K \Tor_i\big(I,S/g(J)\big)_j,
\end{eqnarray*}
as desired,
where $\lambda_1,\lambda_2$ are general elements in $K$.
\end{proof}

\begin{remark}
Let $\ww'=(1,1,\dots,1)$ and note that  the composite of two general changes of coordinates is still general. By replacing $J$ by $h(J)$ for a general change of coordinates $h,$ from Proposition \ref{2.3}(i) it follows that 
\[
\dim_K \Tor_i(S/I,S/h(J))_j \leq \dim_K \Tor_i\big(S/\init_{>_{\sigma}}(I),S/h(J))_j
\]
for any term order $>_\sigma$.

The above fact gives, as a special case, an affirmative answer to \cite[Question 6.1]{Co}.
This was originally proved in the thesis of the first author \cite{Ca2}.
We mention it here because there seem to be no published article which includes the proof of this fact.
\end{remark}

\begin{theorem}
\label{2.5}
Fix an integer $j$.
Let $I \subset S$ and $J \subset S$ be homogeneous ideals.
For a general change of coordinates $g \in \GL_n(K)$,
\begin{itemize}
\item[(i)] $\dim_K \Tor_i(S/I,S/g(J))_j \leq \dim_K \Tor_i(S/\Gin_\lex(I),S/\Gin_\rlex (J))_j
\ \ \mbox{for all }i.$
\item[(ii)] $\dim_K \Tor_i(I,S/g(J))_j \leq \dim_K \Tor_i(\Gin_\lex(I),S/\Gin_\rlex (J))_j
\ \ \mbox{for all }i.$
\end{itemize}
\end{theorem}

\begin{proof}
Without loss of generality,
we may assume $\init_\lex(I)=\Gin_\lex(I)$  and that $\init_\rlex(J)=\Gin_\rlex(J)$. 
It follows from \cite[Propositin 15.16]{Ei} that there are vectors $\ww, \ww' \in \ZZ_{\geq 0}^n$ such that
$\init_\ww(I)=\init_\lex(I)$ and $\init_{\ww'}(g(J))=\Gin_\rlex(J)$.
Then the desired inequality follows from Proposition \ref{2.3}.
\end{proof}

Since $\Tor_0(S/I,S/J)\cong S/(I+J)$ and $\Tor_0(I,S/J)\cong I/IJ$,
we have the next corollary.

\begin{corollary}
\label{2.6}
Let $I \subset S$ and $J \subset S$ be homogeneous ideals.
For a general change of coordinates $g \in \GL_n(K)$,
\begin{itemize}
\item[(i)] $H(I \cap g(J) ,d) \leq H(\Gin_\lex (I)\cap \Gin_\rlex(J),d)$ for all $d \geq 0$.
\item[(ii)] $H(Ig(J),d) \geq H(\Gin_\lex(I)\Gin_\rlex(J),d)$ for all $d \geq 0$.
\end{itemize}
\end{corollary}

We conclude this section with a result regarding the Krull dimension of certain Tor modules. 
We show how Theorem \ref{2.5} can be used to give a quick proof of Proposition \ref{MiSp}, which  is a special case (for the variety $X=\mathbb{P}^{n-1}$ and the algebraic group $\SL_n$) of the main Theorem of  \cite{MS}. 

Recall that generic initial ideals are \textit{Borel-fixed}, that is they are fixed under the action of the Borel subgroup of $\GL_n(K)$ consisting of all the upper triangular invertible matrices. In particular for an ideal $I$ of $S$ and an upper triangular matrix $b\in \GL_n(K)$ one has $b(\Gin_\lex(I))= \Gin_\lex(I).$ Similarly, if we denote by  $op$ the change of coordinates of 
$S$ which sends $x_i$ to $x_{n-i}$ for all $i=1,\dots,n,$ we have that $b( op (\Gin_\rlex(I)))= op (\Gin_\rlex(I)).$

 We call \textit{opposite Borel-fixed} an ideal $J$ of $S$ such that $op(J)$  is Borel-fixed (see  
\cite[\S 15.9]{Ei} for more details on the combinatorial properties of Borel-fixed ideals).

It is easy to see that if $J$ is Borel-fixed, then so is $(x_1,\dots,x_i)+J$ for every $i=1,\dots,n.$ Furthermore if $j$ is an integer equal to $\min \{i : x_i\not \in J \}$ then $J:x_j$ is also Borel-fixed; in this case $I$ has a minimal generator divisible by $x_j$ or  $I=(x_1,\dots,x_{j-1}).$ Analogous statements hold for opposite Borel-fixed ideals.

Let $I$ and $J$ be ideals generated by linear forms. If we assume that $I$ is Borel fixed and that $J$ is opposite Borel fixed, then there exist $1\leq i,j \leq n $ such that $I=(x_1,\dots,x_i)$ and $J=(x_j,\dots,x_n).$  An easy computation shows that the Krull dimension of $\Tor_i(S/I,S/J)$ is always zero when $i>0.$ 

More generally one has

\begin{proposition}[Miller--Speyer]\label{MiSp} Let $I$ and $J$ be two homogeneous ideals of $S.$ For a  general change of coordinates $g$, the Krull dimension of $\Tor_i(S/I,S/g(J))$ is zero for all $i>0.$ 
\end{proposition} 
\begin{proof} When $I$ or $J$ are equal to $(0)$ or to $S$ the result is obvious. Recall that a finitely generated graded module $M$ has Krull dimension zero if and only if $M_d=0$ for all $d$ sufficiently large. In virtue of Theorem \ref{2.5} it is enough to show that $\Tor_i(S/I,S/J)$ has Krull dimension zero whenever $I$ is Borel-fixed, $J$ opposite Borel-fixed and $i>0.$  By contradiction, let the pair $I,J$ be a maximal counterexample (with respect to point-wise inclusion). By the above discussion, and by applying $op$ if necessary, we can assume that 
$I$ has a minimal generator of degree greater than 1. Let $j=\min \{h : x_h\not \in I \}$ and notice that both 
$(I:x_j)$ and $(I+(x_j))$ strictly contain $I.$ For every $i>0$ the short exact sequence $ 0 \rightarrow S/(I:x_j) \rightarrow S/I \rightarrow S/(I+(x_j)) \rightarrow 0$  induces the exact sequence 
\[
\Tor_i(S/(I:x_j),S/J)
\rightarrow 
\Tor_i(S/I,S/J)
 \rightarrow 
\Tor_i(S/(I+(x_j)),S/J).
\]
By the maximality of $I,J$, the first and the last term have Krull dimension zero. Hence the middle term must have dimension zero as well, contradicting our assumption.
\end{proof}

\section{General intersections and general products}

In this section, we prove Theorems \ref{intersection} and \ref{product}. We will assume throughout the rest of the paper $\mathrm{char}(K)=0.$

A monomial ideal $I \subset S$ is said to be \textit{$0$-Borel} (or \textit{strongly stable})
if, for every monomial $u x_j \in I$ and for every $1 \leq i <j$ one has $ux_i \in I$.
Note that $0$-Borel ideals are precisely all the possible Borel-fixed ideals in characteristic $0$.  
In general, the Borel-fixed property depends on the characteristic of the field  and we refer the readers to \cite[\S 15.9]{Ei} for the details. 
A set $W \subset S$ of monomials in $S$ is said to be \textit{$0$-Borel} if the ideal they generate is $0$-Borel, or equivalently if for every monomial $u x_j \in W$ and for every $1 \leq i <j$ one has $ux_i \in W$. Similarly we say that a monomial ideal $J \subset S$ is \textit{opposite $0$-Borel} if for every monomial $ux_j \in J$ and for every $j < i \leq n$ one has $ux_i \in J$.

Let $>_\rev$ be the reverse lexicographic order induced by the ordering $x_1 > \cdots >x_n$.
We recall the following  result \cite[Lemma 3.2]{Mu}.

\begin{lemma}
\label{3-1}
Let $V=\{v_1,\dots,v_s\} \subset S_d$ be a $0$-Borel set of monomials
and $W =\{w_1,\dots,w_s\} \subset S_d$ the lex set of monomials,
where $v_1 \geq_{\rev} \cdots \geq_{\rev} v_s$ and
$w_1 \geq_{\rev} \cdots \geq _{\rev} w_s$.
Then $v_i \geq_{\rev} w_i$ for all $i=1,2,\dots,s$.
\end{lemma}

Since generic initial ideals with respect to $>_\lex$ are $0$-Borel,
the next lemma and Corollary \ref{2.6}(i)
prove Theorem \ref{intersection}.

\begin{lemma}
\label{3-2}
Let $I \subset S$ be a $0$-Borel ideal and $P \subset S$ an opposite lex ideal.
Then $\dim_K(I\cap P)_d \leq \dim_K (I^\lex \cap P)_d$
for all $d\geq 0$.
\end{lemma}

\begin{proof}
Fix a degree $d$.
Let $V,W$ and $Q$ be the sets of monomials of degree $d$ in $I$, $I^\lex$ and $P$ respectively.
It is enough to prove that $\# V \cap Q \leq \# W \cap Q$.

Observe that $Q$ is the set of the smallest $\#Q$ monomials in $S_d$ with respect to $>_\rev$.
Let $m=\max_{>_\rev} Q$.
Then by Lemma \ref{3-1}
$$\# V \cap Q = \# \{ v \in V: v \leq_{\rev} m\} 
\leq \# \{ w \in W: w \leq_{\rev} m\} = \# W \cap Q,$$
as desired.
\end{proof}

Next, we consider products of ideals.
For a monomial $u \in S$, let $\max u$ (respectively, $\min u$) be the maximal (respectively,
minimal) integer $i$ such that $x_i$ divides $u$, where we set $\max 1 = 1$ and
$\min 1 = n$. For a monomial ideal $I \subset S$, let $I_{(\leq k)}$ be the K-vector space spanned by
all monomials $u \in I$ with $\max u \leq k$.

\begin{lemma}
\label{3-4}
Let $I \subset S$ be a $0$-Borel ideal and $P \subset S$ an opposite $0$-Borel ideal.
Let $G(P)=\{u_1,\dots,u_s\}$ be the set of the minimal monomial generators of $P$.
As a $K$-vector space,
$IP$ is the direct sum
$$
IP=\bigoplus_{i=1}^s (I_{(\leq \min u_i)})u_i.
$$
\end{lemma}

\begin{proof}
It is enough to prove that,
for any monomial $w \in IP$, there is the unique expression $w=f(w)g(w)$
with $f(w) \in I$ and $g(w) \in P$
satisfying
\begin{itemize}
\item[(a)] $\max f(w) \leq \min g(w)$.
\item[(b)] $g(w) \in G(P)$.
\end{itemize}
Given any expression $w=fg$ such that $f \in I$ and $g \in P$,
since $I$ is $0$-Borel and $P$ is opposite $0$-Borel,
if $\max f > \min g$ then we may replace $f$ by $f \frac{x_{\min g}} {x_{\max f}} \in I$
and replace $g$ by $g \frac{x_{\max f}} {x_{\min g}} \in P$.
This fact shows that there is an expression satisfying (a) and (b).

Suppose that the expressions $w=f(w)g(w)$ and $w=f'(w)g'(w)$ satisfy conditions (a) and (b).
Then, by (a), $g(w)$ divides $g'(w)$ or $g'(w)$ divides $g(w)$.
Since $g(w)$ and $g'(w)$ are generators of $P$,
$g(w)=g'(w)$.
Hence the expression is unique.
\end{proof}

\begin{lemma}
\label{3-5}
Let $I \subset S$ be a $0$-Borel ideal and $P \subset S$ an opposite $0$-Borel ideal.
Then $\dim_K(IP)_d \geq \dim_K (I^{\lex}P)_d$
for all $d\geq 0$.
\end{lemma}

\begin{proof}
Lemma \ref{3-1} shows that $\dim_K {I_{(\leq k)}}_d \geq \dim_K {I^\lex_{(\leq k)}}_d$
for all $k$ and $d \geq 0$.
Then the statement follows from Lemma \ref{3-4}.
\end{proof}

Finally we prove Theorem \ref{product}.

\begin{proof}[Proof of Theorem \ref{product}]
Let $I'=\Gin_\lex(I)$ and $J'=\Gin_\oplex(J)$.
Since $I'$ is $0$-Borel and $J'$ is opposite $0$-Borel,
by Corollary \ref{2.6}(ii) and Lemmas \ref{3-5}
$$H(Ig(J),d) \geq H(I'J',d) \geq H(I^{\lex} J',d) \geq H(I^\lex J^\rlex,d)$$
for all $d \geq 0$.
\end{proof}

\begin{remark}
\label{rem1}
Theorems \ref{intersection} and \ref{product}
are sharp.
Let $I \subset S$ be a Borel-fixed ideal and $J \subset S$ an ideal satisfying that $h(J)=J$
for any lower triangular matrix $h \in \GL_n(K)$.
For a general $g \in \GL_n(K)$,
we have the LU decomposition $g=bh$
where $h \in \GL_n(K)$ is a lower triangular matrix and $b \in \GL_n(K)$ is an upper triangular matrix.
Then as $K$-vector spaces
$$I \cap g(J) \cong b^{-1}(I) \cap h(J)= I\cap J
\mbox{ and }
I g(J) \cong b^{-1}(I) h(J)= I J.$$
Thus if $I$ is lex and $J$ is opposite lex then
$H(I\cap g(J),d)=H(I\cap J,d)$ and
$H(Ig(J),d)=H(I J,d)$ for all $d\geq 0$.
\end{remark}

\begin{remark}\label{example}
The assumption on $\Gin_\lex(J)$ in Theorem \ref{intersection} is necessary.
Let $I=(x_1^3,x_1^2x_2,x_1x_2^2,x_2^3) \subset K[x_1,x_2,x_3]$
and $J=(x_3^2,x_3^2x_2,x_3x_2^2,x_2^3)\subset K[x_1,x_2,x_3]$.
Then the set of monomials of degree $3$ in $I^\lex$ is
$\{x_1^3,x_1^2x_2,x_1^2x_3,x_1x_2^2\}$
and that of $J^\rlex$ is
$\{x_3^3,x_3^2x_2,x_3^2x_1,x_3x_2^2\}$.
Hence $H(I^\lex\cap J^\rlex,3)=0$.
On the other hand, as we see in Remark \ref{rem1},
$H(I\cap g(J),3)=H(I\cap J,3)=1$. Similarly, the  assumption on the characteristic of $K$ is  needed as one can easily see by considering  $\mathrm{char}(K)=p>0$, $I=(x_1^p,x_2^p)\subset K[x_1,x_2]$ and $J=x_2^p.$ In this case we have 
$H(I^\lex\cap J^\rlex,p)=0$, while $H(I\cap g(J),p)=H(g^{-1}(I)\cap J,p)=1$ since $I$ is fixed under any change of coordinates.
\end{remark}

Since $\Tor_0(S/I,S/J) \cong S/(I+J)$
and $\Tor_1(S/I,S/J) \cong (I\cap J)/ IJ$
for all homogeneous ideals $I \subset S$ and $J \subset S$,
Theorems \ref{intersection} and \ref{product} show the next statement.

\begin{remark}
\label{cor}
Conjecture \ref{conj} is true if $i=0$ or $i=1.$
\end{remark}

\end{document}